\documentclass{article}
\usepackage[utf8]{inputenc}
\usepackage{amsmath}
\usepackage{amssymb}
\usepackage{amsthm}
\usepackage{colonequals}
\usepackage{xcolor}
\usepackage{url}
\usepackage{todonotes}
\usepackage[style=numeric, giveninits]{biblatex}
\DeclareFieldFormat{doi}{}
\DeclareFieldFormat{url}{}
\DeclareFieldFormat{urldate}{}
\DeclareFieldFormat{issn}{}

\newcommand{\CC}{\mathcal{C}}

\newtheorem{theorem}{Theorem}
\newtheorem{observation}[theorem]{Observation}
\newtheorem{lemma}[theorem]{Lemma}
\newtheorem{corollary}[theorem]{Corollary}
\newtheorem{conjecture}[theorem]{Conjecture}

\theoremstyle{definition}

\newtheorem{example}[theorem]{Example}

\title{Towards Characterization of 5-List-Colorability of Toroidal Graphs}
\author{
Zden\v{e}k Dvo\v{r}\'ak\thanks{Computer Science Institute, Charles University, Prague, Czech Republic
E-mail: \protect\nolinkurl{rakdver@iuuk.mff.cuni.cz}.  Supported by project 22-17398S (Flows and cycles in graphs on surfaces) of Czech Science Foundation.}
\and
Félix Moreno Peñarrubia\thanks{Centre de Formació Interdisciplinària Superior (CFIS) - Universitat Politècnica de Catalunya (UPC). Partially supported by the mobility grants program of Centre de Formació Interdisciplinària Superior (CFIS) - Universitat Politècnica de Catalunya (UPC)}
}
\date{\today}

\begin{document}

\maketitle

\begin{abstract}
Through computer-assisted enumeration, we list minimal obstructions
for 5-choosability of graphs on the torus with the following additional
property: There exists a cyclic system of non-contractible triangles around
the torus where the consecutive triangles are at distance at most four.
This condition is satisfied by all previously known obstructions, and we verify that there
are no additional obstructions with this property.  This supports the conjecture
that a toroidal graph is 5-choosable if and only if it is 5-colorable.
\end{abstract}

\section{Introduction}

Thomassen~\cite{thomassentorus} proved that 5-colorability of graphs drawn on the torus
is exactly characterized by four forbidden subgraphs.  That is, he proved that there are only 
four $6$-critical graphs embeddable on the torus (a graph is \emph{$k$-critical} if it is
an inclusionwise-minimal graph of chromatic number at least $k$). Later, he proved
that this result can be generalized to all other surfaces~\cite{thomassenfixedsurface}:
For any fixed surface $\Sigma$, there are only finitely many $6$-critical graphs embeddable on $\Sigma$.

List coloring (or choosability) is a variant of vertex coloring in which, instead of 
choosing the colors of the vertices from a fixed set, each vertex has its own list from 
which its color must be chosen.  A graph is said to be $k$-choosable 
if it is colorable for every list assignment with lists of size $k$. Voigt~\cite{voigt1993} 
exhibited a planar graph that is not $4$-choosable, and Thomassen~\cite{thomassenplanargraphchoosable}
proved that every planar graph is $5$-choosable.

Postle~\cite{postlethesis} proved that analogously to the 5-coloring case,
for any fixed surface $\Sigma$, the $5$-choosability of graphs drawn on $\Sigma$ is
exactly characterized by finitely many obstructions. 
However, the question of finding an explicit characterization 
remains unanswered for all surfaces except for the sphere (where all graphs are $5$-choosable) and the projective plane
(in which the only obstruction to $5$-choosability is $K_6$~\cite{bohmedirac}). In particular, it is not known if there exists any additional 
minimal obstruction for $5$-choosability in the torus beyond the ones that are also 
obstructions for $5$-colorability. 

In this article we develop techniques for the explicit enumeration of critical graphs in 
list-coloring problems, and present the computational results of our implementation. 

\section{Summary of Our Results}

Let us give the definitions needed to state our results precisely.
Let $G$ be a graph. A \emph{list assignment} for $G$ is a function $L : V(G) \rightarrow 2^{\mathbb{N}}$;
it is a \emph{$k$-list assignment} if $|L(v)| \geq k$ for every $v\in V(G)$.
An \emph{$L$-coloring} of $G$ from a list assignment $L$ is a proper coloring $\phi$ such that $\phi(v) \in L(v)$ for every $v\in V(G)$.
A graph $G$ is \emph{$k$-choosable} if $G$ has an $L$-coloring for every $k$-list assignment $L$. 

A graph is $G$ is \emph{$L$-critical} for a list assignment $L$ if $G$ has no $L$-coloring
but every proper subgraph of $G$ has an $L$-coloring.
The graph $G$ is \emph{critical for $k$-list coloring} if there exists a $k$-list assignment $L$ such that
$G$ is $L$-critical.  It is \emph{critical for $k$-choosability} if $G$ is not $k$-choosable,
but every proper subgraph of $G$ is $k$-choosable.  Note that a graph is critical for $k$-choosability
if and only if it is critical for $k$-list coloring but no proper subgraph of $G$ is critical for $k$-list coloring~\cite{onlistcritical}.

If $F$ is a connected graph drawn on the sphere $\Sigma_0$ and $f_1$ and $f_2$ are distinct faces of $F$ bounded by cycles, then we say that $(F,f_1,f_2)$
is a \emph{cylindrical fragment} with \emph{boundary faces} $f_1$ and $f_2$. The \emph{spacing} of the fragment 
is the distance between the cycles bounding $f_1$ and $f_2$ in $F$.  Suppose that $G$ is a graph drawn
on the torus and $C$ is a non-contractible cycle in $G$.  If we cut the torus along $C$ (duplicating the vertices and edges of $C$)
and patch the resulting holes, we obtain a cylindrical fragment with boundary faces formed by the patches.  We call
this cylindrical fragment $(F,f_1,f_2)$ the \emph{$C$-fragment} of $G$.  Note that there exists a natural projection $\pi$
from $\Sigma_0\setminus (f_1\cup f_2)$ to the torus that maps the boundary cycles of $f_1$ and $f_2$ to $C$ and is otherwise injective.

For any cycle $C''$ in $F$ that separates $f_1$ from $f_2$, we say that the cycle $C'=\pi(C'')$ in $G$ is \emph{laminar} with $C$.
Observe in particular that two non-contractible \emph{triangles} in $G$ are laminar if and only if they are homotopically equivalent.
For $i\in\{1,2\}$, let $F_i$ be the subgraph of $F$ drawn between
$f_i$ and $C''$ (including the cycle $C''$) and let $f'_{3-i}$ be the face of $F_i$ bounded by $C''$.  Then $(F_i,f_i,f'_{3-i})$ is
a cylindrical fragment, which we call a \emph{$(C,C')$-fragment} of $G$.  In the case that $C=C'$, we define there to be only one $(C,C')$-fragment, equal to
$(F,f_1,f_2)$.

A \emph{cyclic system of non-contractible triangles} in $G$ is a cyclic sequence $\CC=C_1, \ldots, C_m$ of distinct
laminar non-contractible triangles in $G$ such that for each $i\in\{1,\ldots,m\}$, there is a
$(C_i,C_{i+1})$-fragment $(F,f_1,f_2)$ such that $\pi(F)$ does not contain any of the triangles of $\CC\setminus\{C_i,C_{i+1}\}$
(where $C_{m+1}=C_1$).  We call any such fragment a \emph{$\CC$-fragment}.
Note that there are exactly $m$ $\CC$-fragments.  The \emph{spacing} of the system $\CC$ is the maximum of the spacings
of its $\CC$-fragments.

Our main result is as follows.
\begin{theorem}\label{thm-main}
Suppose that $G$ is a graph drawn on the torus so that $G$ contains a cyclic system of non-contractible triangles of spacing at most four.
Then $G$ is critical for 5-choosability if and only if $G$ is one of the four $6$-critical graphs obtained in~\cite{thomassentorus}.
\end{theorem}

Let us remark that $K_7$ is a graph that can be drawn in the torus which is critical for 
$5$-list coloring, but neither critical for $5$-choosability nor $6$-critical. 
We also verified that $K_7$ is the only graph with this property, i.e., critical for $5$-list coloring,
not critical for 5-choosability, and satisfying the assumptions of Theorem~\ref{thm-main}.

We believe that the cyclic system assumption in Theorem~\ref{thm-main} can be dropped, i.e., that the following conjecture holds.

\begin{conjecture}\label{conj-equiv}
A graph drawn on the torus is critical for 5-choosability if and only if it is $6$-critical.
\end{conjecture}

This is equivalent to the following claim.

\begin{conjecture}
A toroidal graph is $5$-choosable if and only if it is $5$-colorable.
\end{conjecture}

An important step towards this conjecture would be to prove it for graphs without non-contractible triangles.
The \emph{edge-width} of a graph drawn on a surface is the length of
the shortest non-contractible cycle.

\begin{conjecture}\label{conj-ew4}
Every graph drawn on the torus with edge-width at least four is 5-choosable.
\end{conjecture}

Let us remark that Postle~\cite{postle2018hyperbolic} proved that that the edge-width of a graph critical for $5$-list coloring
is at most logarithmic in its genus, and thus Conjecture~\ref{conj-ew4} is certainly true if we replace four by a larger constant.

Let us now describe more technical yet still interesting results that arise in the course of the proof
of Theorem~\ref{thm-main}.  For them, we need the standard concept of criticality relative to a subgraph.
Let $T$ be a subgraph of a graph $G$ and let $L$ be a list assignment for $G$. For an $L$-coloring $\psi$ of $T$,
we say that $\psi$ \emph{extends} to an $L$-coloring of $G$ if there exists an $L$-coloring $\phi$ of $G$
such that $\phi(v) = \psi(v)$ for all $v \in V(T)$.  The graph $G$ is \emph{$T$-critical with respect to $L$}
if $G \neq T$ and for every proper subgraph $G' \subset G$ such that $T \subseteq G'$, there exists an $L$-coloring of
$T$ that extends to an $L$-coloring of $G'$, but does not extend to an $L$-coloring of $G$.
That is, the removal of any vertex or edge of $G$ changes the set of precolorings of $T$ that extend to the rest of the graph.
If the fixed list assignment $L$ is clear from context, then we just say that $G$ is \emph{$T$-critical}.
Note that if $T$ is the null graph, then $G$ is $T$-critical with respect to $L$ if and only if $G$ is $L$-critical.

\subsection{Critical Cycle-Canvases}

Suppose that $F$ is a graph drawn in the closed disk $\Delta$, $C$ is a cycle in $F$ tracing the boundary of the disk,
and $L$ is a list assignment such that $|L(v)| \geq 5$ for every $v\in V(F)\setminus V(C)$.
Then we say that $(F, C, L)$ is a \emph{cycle-canvas}.  The cycle-canvas is \emph{chordless} if $C$ is an induced cycle in $F$.
The \emph{circumference} of $F$ is the length of the cycle $C$.
We say that a cycle-canvas $(F, C, L)$ is \emph{critical} if $F$ is $C$-critical with respect to $L$.

Let $G$ be a graph drawn on a surface,
let $H$ be a subgraph of $G$, and let $f$ be a face of $H$ homeomorphic to an open disk.
Cut the surface along the boundary of $f$ (duplicating the vertices and edges of the boundary walk of $f$ as needed), and let $G'$ be the graph obtained from $G$ in this way.  Note that $G'$ is drawn partly in the original surface with a hole and partly in the closed disk $\Delta$ corresponding to $f$.  Let $F$ be the subgraph of $G'$ drawn in $\Delta$.
Note that the boundary walk of $f$ corresponds to the cycle $C$ in $F$ drawn along the boundary of $\Delta$.
There is a natural projection $\pi$ from $\Delta$ to the closure of $f$ mapping $F$ to the subgraph of $G$ drawn in the closure of $f$, injective everywhere except possibly on $C$.
For any list assignment $L$ for $G$, we define $L_F$ to be the list assignment for $F$ such that $L_F(v)=L(\pi(v))$ for every $v\in V(F)$.
Clearly, if $L$ is a $5$-list assignment, then $G_{f,L}=(F,C,L_F)$ is a cycle-canvas.

The importance of cycle-canvases comes from the following standard result.
\begin{lemma}\label{lemma-cyclecan}
Let $G$ be a graph drawn on a surface, let $T$ and $H$ be subgraphs of $G$, and let $f$ be a face of $H$ homeomorphic to an open disk
such that $T$ is disjoint from $f$.  Let $L$ be a 5-list assignment for $G$.  If $G$ is $T$-critical with respect to $L$
and $f$ is not a face of $G$, then the cycle-canvas $G_{f,L}=(F,C,L_F)$ is critical.
\end{lemma}
\begin{proof}
Let $\pi$ be the natural projection from the disk containing the cycle-canvas $G_{f,L}$ to the closure of $f$ as described
above.  Consider any proper subgraph $Z$ of $F$ containing $C$, and let $Z_0$ be the subgraph of $G$ consisting of $\pi(Z)$
and of the vertices and edges of $G$ not drawn in $f$. Note that $Z_0$ is a proper subgraph of $G$ and that $Z_0$ contains $T$,
since $T$ is disjoint from $f$.  Since $G$ is $T$-critical with respect to $L$, there exists
an $L$-coloring $\psi_0$ of $T$ that extends to an $L$-coloring $\varphi_0$ of $Z_0$, but not to an $L$-coloring of $G$.
Let $\varphi$ be defined by letting $\varphi(v)=\varphi_0(\pi(v))$ for every $v\in V(Z)$, and let $\psi$ be the
restriction of $\varphi$ to $C$.  Observe that $\psi$ is an $L_F$-coloring of $C$ that extends to the $L_F$-coloring $\varphi$ of $Z$.
Moreover, $\psi$ cannot extend to an $L_F$-coloring $\varphi'$ of $F$: Otherwise, let $\varphi'_0(u)=\varphi_0(u)$ for every vertex $u$
of $G$ drawn outside of $f$, and $\varphi'_0(u)=\varphi'(\pi^{-1}(u))$ for every vertex $u$ of $G$ drawn in $f$, and observe
that $\varphi'_0$ is an $L$-coloring of $G$ extending $\psi_0$, which is a contradiction.

Thus, for any proper subgraph $Z$ of $F$ containing $C$, there exists an $L_f$-coloring $\psi$ of $C$ that extends to the $L_F$-coloring $\varphi$
of $Z$, but not of $F$.  We conclude that $F$ is $C$-critical with respect to $L_f$, as required.
\end{proof}

Hence, critical cycle-canvases naturally arise in constructions of critical graphs: When we argue or assume that a graph $G$
critical for 5-list-coloring has a subgraph $H$ with all faces homeomorphic to disks, then $G$ can be obtained from $H$ by pasting
critical cycle-canvases
into faces of $H$.  A fundamental step towards Postle's bounds on graphs critical for 5-list coloring is the following
bound on critical cycle-canvases~\cite{fivelistcoloring2}:

\begin{theorem}\label{thm:linearbound}
If $(G, C, L)$ is a critical cycle-canvas, then $|V(G)| \leq 19 |V(C)|$.
\end{theorem}

This bound means that up to isomorphism, there are only finitely 
many cycle-canvases of any fixed circumference. The upper bound from Theorem~\ref{thm:linearbound}
is of course far from being tight.  As our first step towards Theorem~\ref{thm-main},
we enumerated all chordless critical cycle-canvases of circumference at most $14$;
Table~\ref{tab:cyclecanvases} summarizes the numbers of the cycle-canvases (up to isomorphism)
that we obtained.

\begin{table}[h]

\centering
\begin{tabular}{l | l | l || l | l | l || l | l | l}
$\ell$ & \# & $\max |V|$ & $\ell$ & \# & $\max |V|$ & $\ell$ & \# & $\max |V|$\\
\hline
3 & 0 & -- & 7  & 17    & 11  & 11 & 131220   & 21 \\ 
4 & 0 & -- & 8  & 144   & 14  & 12 & 1447448  & 24 \\
5 & 1 & 6 & 9  & 1259  & 16  & 13 & 16506283 & 26 \\ 
6 & 4 & 9 & 10 & 12517 & 19  & 14 & 193535378       & 29  
\end{tabular}
\caption{Number and size of chordless critical cycle-canvas candidates of circumference $\ell$.}
\label{tab:cyclecanvases}
\end{table}

Table~\ref{tab:cyclecanvases} gives upper bounds on the actual 
numbers, since our program generates a set of \emph{candidates} for 
critical cycle-canvases, that is, it is guaranteed that our lists include
all critical chordless cycle-canvases of the given circumference,
but may also include some non-critical ones.  The enumeration took 113 hours on an Intel Core i7-1195G7 processor.

\subsection{Critical Prism-Canvases}\label{ssec-prism}

Suppose that $(F,f_1,f_2)$ is a cylindrical fragment, the faces $f_1$ and $f_2$ are bounded by triangles $T_1$ and $T_2$,
and $L$ is a list assignment for $F$ such that $|L(v)|\ge 5$ for every $v\in V(F)\setminus V(T_1\cup T_2)$.
Then we say that $(F, T_1, T_2, L)$ is a \emph{prism-canvas}.
We say that a prism-canvas $(F, T_1, T_2, L)$ is \emph{critical} 
if $F$ is $(T_1 \cup T_2)$-critical with respect to $L$.
The next step Postle took on the way to his bounds on critical graphs for 5-list coloring is the following
result on prism-canvases~\cite{postlethesis}.
\begin{theorem}\label{thm:tpp}
There exists an integer $D$ such that every critical prism-canvas has spacing less than $D$.
That is, for every prism-canvas $(F, T_1, T_2, L)$ of spacing at least $D$, every $L$-coloring of $T_1 \cup T_2$
extends to an $L$-coloring of $F$. 
\end{theorem}

For a prism-canvas $(F,T_1,T_2,L)$, let $P$ be a shortest path between $T_1$ and $T_2$ in $F$.  A \emph{skeleton}
of $F$ is the subgraph $H$ of $F$ induced by $V(T_1\cup T_2\cup P)$.  Note that if the distance between $T_1$ and $T_2$
is at least two, then $H$ consists of $T_1\cup T_2\cup P$ and possibly some additional edges between $T_1$ and the
second vertex of $P$ and between $T_2$ and the next to last vertex of $P$; and if the distance is at most one, then
$H$ consist of $T_1$, $T_2$, and edges between them.

By Lemma~\ref{lemma-cyclecan}, every critical prism-canvas can be obtained from its skeleton by pasting
chordless critical cycle-canvases into its faces.  Moreover, if the prism-canvas has spacing $d$, then
the pasted cycle-canvases have circumference at most $2d+6$.  Hence, using our list of chordless critical cycle-canvas candidates,
we can obtain all candidates for critical prism-canvases of spacing at most four.
We performed this procedure and eliminated the prism-canvases that we were able to show to be non-critical from the list.
In this way, we obtained $510$, $2719$, $300$ and $5$ candidates for critical prism-canvases of spacing
$1$, $2$, $3$ and $4$, respectively.

\begin{figure}
\centering
\includegraphics{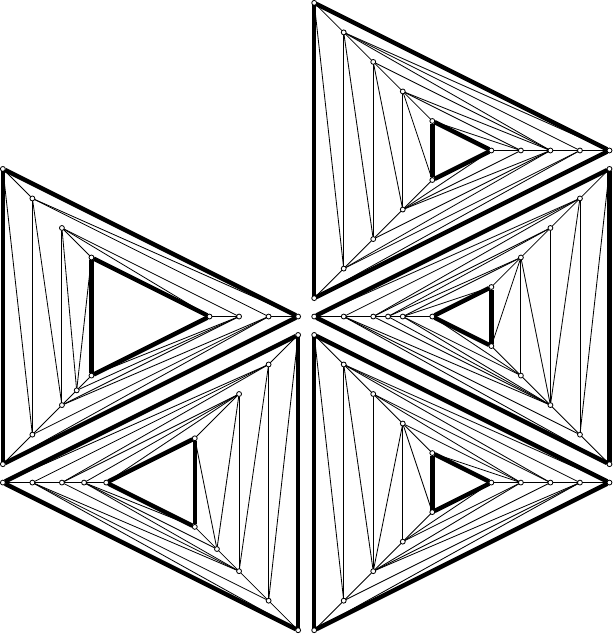}
\caption{Candidates for critical prism-canvases of spacing 4.}\label{fig-canvases}
\end{figure}

The value of $D$ for which Postle proves Theorem~\ref{thm:tpp} is not made 
explicit, but it can be estimated to be around $100$. The rapid decrease of the number of critical prism-canvases suggests
that $D=5$ may be correct.
\begin{conjecture}\label{con:tpp}
There is no critical prism-canvas of spacing at least five.
\end{conjecture}

To better test this conjecture, we pushed the enumeration a bit further.
Consider a prism-canvas $Q=(F,T_1,T_2,L)$, and let $C$ be a triangle in $Q$ distinct from $T_1$ and $T_2$
separating the boundary faces.  For $i\in\{1,2\}$, let $Q_i=(F_i,T_1,C,L)$ be the prism-canvas with $F_i$
being the subgraph of $F$ drawn between $T_1$ and $C$.
Analogously to Lemma~\ref{lemma-cyclecan}, it is easy to see that if $Q$ is critical,
then the prism-canvases $Q_1$ and $Q_2$ are also critical.  Hence, every critical prism-canvas with a triangle separating its boundary face
can be obtained by gluing smaller critical prism-canvases.  We tried gluing together the prism-canvases from the list
we obtained above, and after eliminating the provably non-critical ones, we found no new prism-canvases.
This gives us a slight strengthening of Theorem~\ref{thm:tpp}; let us give the definition necessary to state it.

A \emph{linear system of separating triangles} in a prism-canvas $Q=(F,T_1,T_2,L)$
is a sequence $T_1=C_0, C_1, \ldots, C_m=T_2$ of distinct triangles in $F$ separating the boundary faces,
such that for $i\in\{1,\ldots,m-1\}$, the triangle $C_i$ separates $C_{i-1}$ from $T_2$.
For $i\in\{1,\ldots,m\}$, let $Q_i=(F_i,C_{i-1},C_i,L)$ be the prism-canvas with $F_i$ consisting of the subgraph of $F$ between
$C_{i-1}$ and $C_i$.  The \emph{spacing} of the system is the maximum of the spacings of the prism-canvases $Q_1$, \ldots, $Q_m$.

\begin{theorem}\label{thm:prismcanvasresults}
Let $Q$ be a prism-canvas containing a linear system of separating triangles of spacing at most four.  If the prism-canvas $Q$ is critical, then $Q$
has spacing at most four.
\end{theorem}

\subsection{Critical Graphs on the Torus}\label{ssec-torus}

Suppose that $G$ is a graph drawn on the torus critical for 5-list coloring, and let $T$ be
a non-contractible triangle in $G$.  Let $(F,f_1,f_2)$ be the $T$-fragment of $G$,
and let $T_1$ and $T_2$ the triangles bounding $f_1$ and $f_2$.  Analogously to 
Lemma~\ref{lemma-cyclecan}, there exists a 5-list assignment $L$ such that $Q=(F,T_1,T_2,L)$ is a critical prism-canvas.
Moreover, if $T$ is contained in a cyclic system of non-contractible triangles of spacing $d$ in $G$, then $Q$ has
a linear system of separating triangles of spacing at most~$d$.

By Theorem~\ref{thm:prismcanvasresults}, the following claim holds.
\begin{corollary}
Let $G$ be a  graph drawn on the torus critical for 5-list coloring, and let $T$ be a non-contractible triangle in $G$.  If $T$ is contained in a cyclic system of non-contractible triangles of spacing at most four, then the $T$-fragment of $G$ is a critical prism-canvas of spacing at most four.
\end{corollary}
Thus, every graph on the torus which is critical for 5-list coloring and contains a cyclic system of non-contractible triangles of spacing at most four can be obtained by gluing the ends of a critical prism-canvas of spacing at most four.
We performed this procedure for the prism-canvases from our list obtained
in Section~\ref{ssec-prism}; all of the resulting toroidal graphs are either provably 5-choosable, or contain one of the four 6-critical toroidal
graphs obtained by Thomassen~\cite{thomassentorus} as subgraphs.  This finishes the proof of Theorem~\ref{thm-main}.

Moreover, this shows that Conjecture~\ref{con:tpp} implies Conjecture~\ref{conj-equiv} for graphs of edge-width three, i.e., the ones with
a non-contractible triangle.  Hence, to prove Conjecture~\ref{conj-equiv}, one needs to prove Conjectures~\ref{conj-ew4} and \ref{con:tpp}.

\section{Techniques}

In this section we will explain the computational procedures used to obtain the above 
results. There are two main aspects to explain:

\begin{enumerate}
	\item How do we generate (candidates for) critical cycle-canvases.  We have already explained how
	we obtain candidates for critical prism-canvases and critical graphs on
	the torus from them in Sections~\ref{ssec-prism} and~\ref{ssec-torus}.
	\item How do we test canvases and graphs on the torus for criticality. 
\end{enumerate}

In our computer program, we represents the canvases as plane graphs encoded by the rotation systems around vertices, and \emph{without}
list assignments.  For example, for cycle-canvases we store plane graphs $G$ with the outer face bounded by a cycle $C$ such that there 
possibly exists a list assignment $L$ so that $(G, C, L)$ is a critical cycle-canvas.  This sidesteps the issue that
the number of such lists assignments is typically quite large.

In order to identify isomorphic plane graphs, we compute a canonical form
for each graph: For each edge of the outer face, we compute a transcript of the DFS traversal of the graph
starting at that edge, in which edges exiting each vertex are processed in the order specified by the
rotation system (starting next to the edge through which we first arrived to the vertex).
We assign numbers to vertices in the order in which they are first visited, and as the transcript, we
print the vertex number each time it is revisited.  The canonical form is chosen to be the 
lexicographically smallest transcript among those computed. 

The implementation of the algorithms described in this section can be found at \url{https://github.com/FelixMorenoPenarrubia/5ListColorabilityOfToroidalGraphs}.

\subsection{Generating Critical Cycle-Canvases}

Our generation procedure for critical cycle-canvases is based on the following result:

\begin{theorem}[Cycle Chord or Tripod Theorem~\cite{fivelistcoloring2}]
\label{thm:cyclechordtripod}
If $(G, C, L)$ is a critical cycle-canvas, then either

\begin{enumerate}
\item $C$ has a chord in $G$, or
\item there exists a vertex $v \in V(G) \setminus V(C)$ with at least three neighbors 
on $C$ such that at most one of the faces of $G[\{v\} \cup V(C)]$ contains a vertex or 
an edge of $G$. 
\end{enumerate}
\end{theorem}

This theorem together with Lemma~\ref{lemma-cyclecan} says that every critical cycle-canvas can either be decomposed into two smaller critical 
cycle-canvases through a chord in the outer face or it can be decomposed into a ``tripod'', a vertex $v$ 
with at least three neighbors in $C$, and possibly a smaller critical cycle-canvas contained in the only nonempty 
face incident with $v$.

Note that if the cycle-canvas $(G, C, L)$ is chordless, then in the second case, the cycle-canvas
in the non-empty face of $G[\{v\} \cup V(C)]$ is also chordless.
This implies that we can generate all critical chordless cycle-canvases by starting with the trivial cycle-canvas
and repeatedly adding tripods to the outside of the cycle-canvas.  Note that the cycle-canvases arising in this
way are not guaranteed to be critical, and we need to check this condition separately.  Because of Lemma~\ref{lemma-cyclecan},
once a cycle-canvas is shown not to be critical, we do not need to derive further cycle-canvases from it by addition of tripods.

If the tripod $(G, C, L)$ has circumference $\ell$ and the vertex $v$ in the conclusion of Theorem~\ref{thm:cyclechordtripod}
has at least four neighbors in $C$, or three neighbors that are not consecutive, then the faces of $G[\{v\} \cup V(C)]$
have length less than $\ell$.  Hence, if we already generated all critical chordless cycle-canvases of circumference less than
$\ell$, then $(G, C, L)$ can be obtained from one of them by the addition of a single tripod.  However, if the vertex $v$
is adjacent to three consecutive vertices of $C$, then the smaller cycle-canvas has the same circumference $\ell$.

In order to resolve this, we first generate all the cycle-canvases obtained from cycle-canvases of smaller circumference,
enqueue the resulting critical canvases, and then process the canvases from the queue and add tripods 
to three consecutive vertices in all possible ways, enqueueing the new potentially critical cycle-canvases that are found.

\subsection{Testing Graph and Canvas Criticality}

In this section, we explain how we test graphs and canvases for criticality. 
Let $K$ be a graph and let $s : V(K) \rightarrow \mathbb{N}$ be a function. We say that a list
assignment $L$ is an \emph{$s$-list assignment} if $|L(v)| \geq s(v)$ for all $v \in V(K)$. 
We say that $K$ is \emph{$s$-colorable} if $K$ is $L$-colorable for every $s$-list assignment $L$.
We say that $K$ is \emph{$s$-reducible} if there exists a proper subgraph $H \subsetneq K$ so that
for every $s$-list assignments $L$ to $K$, if $H$ is $L$-colorable then $K$ is 
$L$-colorable.  Clearly, if $K$ is $s$-colorable, then $K$ is also $s$-reducible,
since we can take $H$ to be the null graph.  We say that the graph $K$ is \emph{$s$-irreducible} if it is not $s$-reducible.

\begin{example}
\label{ex:wheel}
Consider the $5$-wheel $W$ with the central vertex $u$ and the rim $v_1\ldots v_5$.
Let $s(u) = 5$, $s(v_1) = s(v_2) = 3$, and $s(v_3) = s(v_4) = s(v_5) = 2$.
Then $W$ is $s$-colorable, and therefore $s$-reducible. 
\end{example}

\begin{example}
\label{ex:reduciblenoncolorable}
Let $K$ be the graph with $V(K) = \{u, v, w, x, y\}$ and $E(K) = \{uv, uw, ux, uy, vw, wx, xy\}$,
and let $s(u) = 4$ and $s(v) = s(w) = s(x) = s(y) = 2$.  Then $K$ is $s$-reducible, as can be seen by
considering the subgraph $H=K-wx$.
\end{example}

As we have discussed in the previous sections, given a graph $G$, a proper subgraph $T$ of $G$, and a function
$s_G:V(G)\to\mathbb{N}$ specifying the list size at each vertex of $G$, we need a conservative $T$-criticality test,
i.e., a procedure that answers ``yes'' if there exists an $s_G$-list assignment $L$ such that $G$ is $T$-critical with respect to $L$.
If such an $s_G$-list assignment does not exist, the answer can be either ``yes'' or ``no'', though preferably
it should have only a small false positive rate (a procedure that always answers ``yes'' would technically be correct, but
obviously useless).

Because of the following observation, we can test for irreducibility, instead.  For a subgraph $T$ of a graph $G$
and a vertex $v\in V(G)\setminus V(T)$, let $d_T(v)$ be the number of neighbors of $v$ in $T$.
Given a list assignment $L$ for $G$, let $s_{G,T,L}:V(G)\setminus V(T)\to \mathbb{Z}$ be defined by setting $s_{G,T,L}(v)=|L(v)|-d_T(v)$
for every $v\in V(G)\setminus V(T)$.

\begin{observation}\label{obs-critir}
Let $T$ be a proper subgraph of a graph $G$.
If $G$ is $T$-critical with respect to a list assignment $L$, then the subgraph $G-V(T)$ is $s_{G,T,L}$-irreducible.
\end{observation}
\begin{proof}
Suppose for a contradiction that $G-V(T)$ is $s$-reducible.  Hence, there exists a proper subgraph $H \subsetneq G-V(T)$
such that for every $s_{G,T,L}$-assignment $L'$ to $G-V(T)$, if $H$ is $L'$-colorable, then $G-V(T)$ is $L'$-colorable.
Let $H'$ be the proper subgraph of $G$ consisting of $G[V(T)]$, $H$, and all edges of $G$ between them.
Since $G$ is $T$-critical with respect to $L$, there exists an $L$-coloring $\psi_T$ of $T$ which extends to an $L$-coloring $\psi'$ of $H'$,
but not of $G$.

Let $L'(v) = L(v) \setminus \{\psi_T(u) : u \in V(T), uv \in E(G)\}$ for every $v\in V(G)\setminus V(T)$.
Then $L'$ is an $s_{G,T,L}$-list-assignment for $G-V(T)$ and $\psi'$ restricts to an $L'$-coloring of
$H$.  Thus, there exists an $L'$-coloring $\psi$ of $G-V(T)$.  However, then $\psi_T\cup \psi$ is an $L$-coloring of $G$ extending $\psi_T$,
since the choice of $L'$ ensures that the colors on the edges between $V(T)$ and $V(G)\setminus V(T)$ do not clash.
This is a contradiction.
\end{proof}

Therefore, to test whether a cycle-canvas or a prism-canvas could be critical, we remove the boundary subgraph $T$,
let $s(v) = 5 - d_T(v)$ for every remaining vertex $v$, and then test the resulting graph $F$ for $s$-reducibility. 
If $F$ turns out to be $s$-reducible, we know that the original canvas is not critical, and we can discard it.
To prove $s$-reducibility of a graph $G$, we note that irreducibility is hereditary in the following sense.
For a subgraph $S$ of a graph $G$ and a function $s:V(G)\to\mathbb{N}$, let us define
$s_{G,S}(v)=s(v)-d_S(v)$ for every $v\in V(G)\setminus V(S)$.

\begin{observation}\label{obs-irrsg}
Let $G$ be a graph and $s:V(G)\to\mathbb{N}$ be a function.  Let $G'$ be an induced subgraph of $G$ and let $S=G-V(G')$.
If $G$ is $s$-irreducible, then $G'$ is $s_{G,S}$-irreducible.
\end{observation}
\begin{proof}
Suppose for a contradiction that $G'$ is $s_{G,S}$-reducible.
Hence, there exists a proper subgraph $H' \subsetneq G'$
such that for every $s_{G,S}$-assignment $L'$ to $G'$, if $H'$ is $L'$-colorable, then $G'$ is $L'$-colorable.
Let $H$ be the proper subgraph of $G$ consisting of $S$, $H'$, and all edges of $G$ between them.
Since $G$ is $s$-irreducible, there exists an $s$-list assignment $L$ to $G$ such that $H$ is $L$-colorable
but $G$ is not $L$-colorable.

Let $\psi$ be an $L$-coloring of $H$, and for every $v\in V(G')$, let $L'(v)=L(v)\setminus \{\psi(u):u\in V(S),uv\in E(G)\}$.
Then $L'$ is an $s_{G,S}$-assignment for $G'$, and $\psi$ restricts to an $L'$-coloring of $H'$.
Thus, $G'$ has an $L'$-coloring $\psi'$.  Then $(\psi\restriction V(S))\cup \psi'$ is an $L$-coloring of $G$, since the
choice of $L'$ ensures that the colorings do not clash on the edges between $V(G')$ and $S$.
This is a contradiction, since $G$ is not $L$-colorable.
\end{proof}

Suppose that $T$ is a proper subgraph of a graph $G$ and $s_G:V(G)\to\mathbb{N}$ is a function specifying the list size at each vertex of $G$.
If we find an induced subgraph $G'$ of $G-V(T)$ which is $s_{G,G-V(G')}$-reducible,
then by Observations~\ref{obs-critir} and \ref{obs-irrsg}, $G$ is not $T$-critical with respect to any $s$-list-assignment,
and thus our procedure can correctly answer ``no''.

\subsubsection{The Alon-Tarsi Method}

In order to systematically detect $s$-colorable (and thus $s$-reducible) configurations,
we use the following result by Alon and Tarsi~\cite{alontarsi}, which provides an efficiently testable sufficient condition.
A subgraph $\vec{H}$ of a directed graph $\vec{G}$ is \emph{Eulerian} if the outdegree of each vertex of $\vec{H}$ is the
same as its indegree, and \emph{spanning} if $V(\vec{H})=V(\vec{G})$.  A directed graph $\vec{H}$ is \emph{even} if it has even number of edges, and \emph{odd} otherwise.

\begin{theorem}
\label{thm:alontarsi}
Let $\vec{G}$ be an orientation of a graph $G$, and let $L$ be an assignment of lists to vertices of $G$ such that 
$|L(v)| > \deg_{\vec{G}}^+(v)$ for every $v\in V(G)$. If $\vec{G}$ has a different number of even and odd spanning Eulerian subgraphs,
then $G$ has an $L$-coloring.
\end{theorem}

There are more efficient ways to implement the Alon-Tarsi method than a naive 
simulation of the above statement; for a thorough study of the implementation
and efficiency of the Alon-Tarsi method in practice, see~\cite{dvorakefficientalontarsi}.
We need the following obvious consequence of this theorem.
\begin{corollary}
Let $K$ be a graph and $s:V(K)\to \mathbb{N}$ a function.  Let $\vec{K}$ be an orientation
of $K$ such that $s(v)>\deg_{\vec{K}}^+(v)$ for every $v\in V(K)$.  If $\vec{K}$ has a different number of even
and odd spanning Eulerian subgraphs, then $K$ is $s$-colorable.
\end{corollary}

Let us now return to the considered situation, where we are given a graph $G$, a proper subgraph $T$ of $G$, and a function
$s_G:V(G)\to\mathbb{N}$ specifying the list size at each vertex of $G$, and we look for reducible induced subgraphs of $G-V(T)$.
We need to decide to which induced subgraphs we apply the Alon-Tarsi test, since applying it to all induced subgraphs of $G-V(T)$
is prohibitively time-consuming.  We proceed as follows: We start by letting $G_0=G-V(T)$.
If the Alon-Tarsi test returns a positive answer for $G_0$ and $s_0=s_{G,G-V(G_0)}$, we are done.
If the answer is negative, the issue often is that $G_0$ has a small proper subgraph $B_0$ which is not $s_0$-colorable---we
want to exclude this subgraph from consideration and look for a reducible induced subgraph only in the complement of $B_0$.
The same issue may of course arise in $G_0-V(B_0)$, and thus we iterate this procedure.

More precisely, for $i=0,1,\ldots$, if $G_i$ is non-empty, we perform the Alon-Tarsi test for $G_i$ and $s_i=s_{G,G-V(G_i)}$.
If the answer is positive, the induced subgraph $G_i$ is $s_i$-colorable, and thus $s_i$-reducible, and we answer ``no''.
If the answer is negative, then let $B_i$ be an induced subgraph of $G_i$ obtained by greedily deleting vertices
whose removal preserves the negative answer to the $s_i$-colorability Alon-Tarsi test, and let $G_{i+1}=G_i-V(B_i)$.
If an empty graph $G_i$ is reached, the procedure fails and we answer ``yes''.

Let us remark that when looking for $B_i$, the function $s_i$ specifying the list sizes stays the same,
i.e., we do not replace it with $s_{G,G-V(B_i)}$.  
Overall, the results of this heuristic turned out to be very satisfactory. 

The Alon-Tarsi test does not always detect colorable graphs successfully. 
For example, the graph in Example~\ref{ex:wheel} is not recognized by the Alon-Tarsi test as being $s$-colorable.
And we also have non-colorable reducible configurations (as in Example \ref{ex:reduciblenoncolorable}), for which Alon-Tarsi does not apply. 
Before performing the Alon-Tarsi based test
as described above, our procedure tests the presence of these two special configurations, and if one
is found, we answer ``no'' directly.

Let us remark that for bigger graphs, the Alon-Tarsi test can be rather slow,
and in particular it is slower than the search for fixed small induced subgraphs.
Hence, it is advisable to execute less expensive tests before running the Alon-Tarsi test,
which may include searching for a list of small, fixed reducible configurations, or running simple
colorability heuristics.

A heuristic that we have used in our implementation tries
to greedily construct a coloring by making coloring choices that are valid for any $s$-list-assignment;
i.e., in general, when we decide to assign a color to a vertex $v$, we decrease $s$ by one at all
adjacent vertices, thus simulating the possibility that the color given to $v$ appeared in their lists.
However, when possible, we make use of the following two observations to eliminate this decrease:

\begin{itemize}
    \item If $u$ and $v$ are neighbors with $|L(v)| > |L(u)|$, then we can color
    $v$ with a color not in the list of $u$, hence not decreasing the number of
    available colors for $u$.
    \item If $u, v, w$ are vertices such that $u$ and $v$ are neighbors and $u$ and $w$ are neighbors but $v$ and $w$ are not neighbors, and $|L(v)| + |L(w)| > |L(u)|$,
    then either $L(v)$ and $L(w)$ are disjoint, and hence we can color one of $v$
    and $w$ with a color not in $L(u)$; or we can color $v$ and $w$ with the same color. In either case, the number of available colors for $u$ can only decrease by one after coloring both $v$ and $w$. 
\end{itemize}

This heuristic is quite fast and it relatively frequently detects $s$-colorability of larger subgraphs, for which
performing the Alon-Tarsi test would be very slow.

\section{Acknowledgments}

The results of this paper form part of the bachelor thesis of the second author, under the guidance of the first author. The second author would like to thank O. Baeza, E. Moreno, J. Pérez, X. Povill and M. Prat for sharing some computation time. 

\printbibliography

\end{document}